\documentclass[a4paper,12pt]{article}

\usepackage[T1]{fontenc}

\usepackage{amsmath,amsthm}
\usepackage{amssymb}
\usepackage{amscd}
\usepackage[latin1]{inputenc}
\usepackage{amsfonts}
\usepackage{multicol}
\usepackage[matrix,arrow,curve]{xy}
\usepackage{authblk}
\usepackage{mathrsfs}
\usepackage{tikz-cd}

\usepackage{graphicx}
\usepackage{changebar,color}
\usepackage[OT2,T1]{fontenc}

\theoremstyle{plain}
\newtheorem{prop}{Proposition}[section]
\newtheorem{thm}[prop]{Theorem}
\newtheorem{lemma}[prop]{Lemma}
\newtheorem{coro}[prop]{Corollary}

\theoremstyle{definition}
\newtheorem{defi}[prop]{Definition}
\newtheorem{rema}[prop]{Remark}
\newtheorem{ex}[prop]{Example}

\DeclareSymbolFont{cyrletters}{OT2}{wncyr}{m}{n}
\DeclareMathSymbol{\Sha}{\mathalpha}{cyrletters}{"58}

\title{Cubic surface bundles and the Brauer group}

\author{Alena Pirutka}

\begin{document}

\maketitle

\begin{abstract}
In this note we define a subgroup $H^i_{nr,\pi}$ of unramified cohomology group $H^i_{nr}$ of a fibration $\pi:X\to S$. This subgroup can be used efficiently in refined specialization arguments and allows to 
detect the failure of stable rationality of a variety specializing to $X$. We compute  $H^2_{nr, \pi}$ systematically for many cubic surface bundles $\pi:X\to S$ over a smooth projective rational surface over an algebraically closed field: we give a combinatorial formula in terms of components of the discriminant divisor of $\pi$.
\end{abstract}

\section{Introduction}
Let $k$ be an algebraically closed field  and let $X$ be a projective variety over $k$. Recall that $X$ is {\it rational} if it is birational to a projective space $\mathbb P^n_k$, and that $X$ is {\it stably rational} if $X\times \mathbb P^m_k$ is rational, for some $m$. In recent years there has been a lot of progress understanding these properties, namely, for various classes of algebraic varieties  it was established that a  very general variety in this class is not stably rational. This includes a large class of hypersurfaces, cyclic covers, complete intersections, fibrations in conics, hypersurfaces in $\mathbb P^n_k\times \mathbb P^m_k$, and other examples (some of these results are in \cite{ABP,AO,B1,ABP,CTP,HKT,HPT1,SWM,NO,O,T,Sch2,Sch1,Sch,V1}).

These results were obtained using  specialization techniques  (see \cite{V1, CTP, T, Sch1,Sch}) to degenerate the variety of interest  to some mildly singular {\it reference} variety $X$: a general idea is that if $X$ has suitable nontrivial birational invariants, then the variety of interest is not stably rational.  In applications, one uses unramified cohomology groups $H^i_{nr}$, properties of differential forms in positive characteristic, or other invariants. 

In practice, both finding degenerations and reference varieties proved to be difficult.  Some reference varieties (in particular the quadric surface bundle in \cite{HPT1} and more recently quadric bundles over $\mathbb P^n$ in \cite{Sch1}) were reused to establish the failure of  stable rationality of different classes of varieties. 

Several known examples of reference varieties come with a fibration structure: one uses fibrations in conics and  quadric surfaces, fibrations in quadrics with generic fibre a Pfister (or a Pfister neighbor) quadric or a Fermat-Pfister form. 

 In this note we introduce the {\em relative unramified cohomology group} $$H^i_{nr,\pi}(k(X)/k, \mu_n^{\otimes j})$$ for a fibration $\pi:X\to S$,  where $S$ is a smooth and projective variety over $k$, and $\pi$ is a projective morphism (see Definition \ref{hnrpi} for details) with integral generic fibre.
 The definition of this group combines the well-known general method of Colliot-Th\'el\`ene and Ojanguren \cite{CTO} to construct unramified classes for fibrations with recent refinement  of Schreieder \cite{Sch2, Sch1, Sch}, which allows to avoid any analysis of singularities of $X$ in specialization arguments at the cost of additional restrictions on unramified classes.  We use slightly stronger restrictions for the group $H^i_{nr,\pi}$, which are however often satisfied in practice (for example, in \cite{HPT1}). 
In particular, if the generic fibre of $\pi$ is smooth, and if the group  $H^i_{nr,\pi}$ is nontrivial, we show that $X$  is a reference variety (no restrictions on singularities of $X$ are needed).

As an example, we consider hypersurfaces of bidegree $(2,3)$ in $\mathbb P^2_k\times \mathbb P^3_k$, where the fibration $\pi$ is defined by the projection on $\mathbb P^2_k$. These cubic surface bundles were studied in \cite{ABP}.  In particular,  we know that the only possible nontrivial classes in $H^2_{nr,\pi}$ are $3$-torsion classes.  In addition, some fibres of $\pi$  must be unions of three planes permuted cyclically  by Galois. 
Here we extend this discussion: let $K$ be the field of functions of $S$ and let $X_K$ be the generic fibre of $\pi$. If $x\in S$ is a point of $S$, distinct from the generic point, let $K_x$ be the field of fractions of the completed local ring $\widehat{O}_{S,x}$ of $S$ at $x$. Instead of analyzing  singular fibres of $\pi$ we are reduced to understanding the minimality of the cubic surface, over the set $\{K_x\}_{x\in S}$ of overfields of $K$.  

This approach allows us to find some concrete reference varieties, for example   (see
\ref{refv}):
 $X\subset \mathbb P^2_{[x:y:z]}\times \mathbb P^3_{[u:v:w:t]}$ a cubic surface bundle over an algebraically closed field $k$ given by equation
$$
xz^2u^3+y^2zv^3+xy^2w^3+ft^3=0,
$$
where $f=x^3+y^3+z^3+3x^2y+3xy^2+3y^2z+3yz^2+3xz^2+3x^2z.$

\

More generally, we determine the group
$$
H^2_{nr,\pi}(k(X)/k, \mathbb Z/3\mathbb Z)
$$
  for a large class of cubic surface bundles $\pi:X\to S$
 (see Theorem \ref{formula}) in terms of a combinatorial formula on particular points of $S$. More precisely,  we consider $x\in S$ such that over the field $K_x$ the surface $X_{K_x}$ is birational to a Severi-Brauer surface.  This extends previous examples of complete combinatorial formulas for unramified cohomology of fibrations in conics, quadric surfaces, involution surface bundles, and Brauer-Severi bundles   \cite{CT15,KT1,KT2,P}.

The paper is organized as follows: in section \ref{sstrategy} we review the general strategy to construct unramified classes for fibrations and we apply it to cubic surface bundles,  in section \ref{sff} we discuss the example above, and in section \ref{sf} we give the general formula.

\subsubsection*{Notations and reminders.} If $k$ is a field, we denote by $k^*$ the set of nonzero elements of $k$. We denote by $\zeta_n$ a primitive $n^{\mathrm{th}}$ root of unity. We assume that $n$ is invertible on $k$.

\paragraph{Galois cohomology.} We denote by $\mu_{n}$ the \'etale $k$-group scheme of the $n^{th}$ roots of unity. For $j$ a positive integer we denote $\mu_{n}^{\otimes j}=\mu_{n}\otimes\ldots\otimes\mu_{n}$ ($j$ times).  If $j<0$, we set $\mu_{n}^{\otimes j}=Hom_{k-gr}(\mu_{n}^{\otimes (-j)}, \mathbb Z/n)$ and $\mu_{n}^{\otimes 0}=\mathbb Z/n$.  If $\zeta_n\in k$  we have an isomorphism   $\mu_{n}^{\otimes j}\stackrel{\sim}{\to}\mathbb{Z}/n$ for any~$j$. 

We denote by $H^i(k, \mu_n^{\otimes j})$ the Galois cohomology groups. Recall that the Kummer theory gives isomorphisms $H^1(k, \mu_n)=k^*/k^{*n}$ and $H^2(k, \mu_n)=\mathrm{Br}\,k[n]$, the $n$-torsion in the Brauer group of $k$. Let $K/k$ be a cyclic Galois extension with Galois group $G=\mathbb Z/n\mathbb Z$. One has $H^1(G, \mathbb Z/n\mathbb Z)=\mathrm{Ker}[H^1(k, \mathbb Z/n\mathbb Z)\to H^1(K, \mathbb Z/n\mathbb Z)]$, and one  associates to the extension $K/k$  a class $[K/k]\in H^1(k, \mathbb Z/n\mathbb Z)$, corresponding to the generator of the group  $H^1(G, \mathbb Z/n\mathbb Z)=Hom(G,\mathbb Z/n\mathbb Z)=\mathbb Z/n\mathbb Z$.

Recall that a Severi-Brauer variety over $k$ is a smooth and projective variety $X$ over $k$, such that over the algebraic closure $\bar k$ of $k$ one has an isomorphism $X_{\bar k}\simeq \mathbb P^n_{\bar k}$.   One associates to $X$ a class $[X]\in \mathrm{Br}\, k$, such that $[X]=0$ if and only if $X$ splits over $k$: $X\simeq \mathbb P^n_{k}$.  Here we will be interested in Severi-Brauer surfaces $S$, one then has $[S]\in H^2(k, \mu_3)$. 

\paragraph{Residues.} If $K$ is the field of fractions of a discrete valuation ring $A$ 
one defines the residue maps  
\begin{equation}\label{defres}
\partial_v^i: H^i(K, \mu_{n}^{\otimes j})\to H^{i-1}(\kappa(v), \mu_{n}^{\otimes (j-1)}),
\end{equation}
where $\kappa(v)$ is the residue field.
If $A_v$ is the completion of  $A$ and $K_v$ is the field of fractions of $A_v$,  then the residue map factorizes through $K_v$.

In this text we only consider discrete valuations (of rank one). We will use notations $\partial^i_v$ or $\partial_v$ for the residue map. 

In degree one we have: if $a\in H^1(K, \mu_n)=K^*/K^{*n}$, then $$\partial^1_v(a)=v(a)\mbox{ mod } n.$$ In degree two, if $(a,b):=a\cup b\in H^2(K, \mu_n)$,then  
\begin{equation}\label{resdeg2}\partial^2_v(a, b)=(-1)^{v(a)v(b)}\overline{\frac{a^{v(b)}}{b^{v(a)}}},
\end{equation} 
where $\overline{\frac{a^{v(b)}}{b^{v(a)}}}$ is the image of the unit $\frac{a^{v(b)}}{b^{v(a)}}$ in $\kappa(v)^*/\kappa(v)^{*n}$.

We refer to \cite{CTO,P} and references therein for more details. We will use the following exact sequence, where $i>0$ (see \cite{CT}, (3.10)):
\begin{equation}\label{gersten}
0\to H^i_{\acute{e}t}(A, \mu_{n}^{\otimes j}))\to H^i(K, \mu_{n}^{\otimes j}))\stackrel{\partial}\to H^{i-1}(\kappa(v), \mu_{n}^{\otimes (j-1)}))\to 0.
\end{equation}

\paragraph{Unramified cohomology.} Let $K=k(X)$ be the function field of an integral algebraic variety $X$. The unramified cohomology groups of $X$ are defined as 
$$H_{nr}^i(k(X)/k, \mu_n^{\otimes j})=\bigcap\limits_v\mathrm{Ker}[H^i(k(X), \mu_n^{\otimes j})\stackrel{\partial_{v}^i}{\to}H^{i-1}(\kappa(v), \mu_n^{\otimes j-1})],$$ where the intersection is over all discrete valuations $v$ on  $k(X)$ (of rank one), trivial on the field $k$.
If $v$ corresponds to a codimension $1$ point $x$ of $X$ or an irreducible divisor $D\subset X$ we will also write $\partial_D$  or $\partial_x$ for the residue map $\partial_v$.

If $X$ is an algebraic variety over a field $k$, the Brauer group of $X$ is defined by $\mathrm{Br}\, X=H^2_{\acute{e}t}(X, \mathbb G_m)$. If $X$ is smooth,  the natural map 
\begin{equation}\label{injbr}
\mathrm{Br}\,X\to \mathrm{Br}\, k(X)
\end{equation}
is injective (see \cite{GB} 1.10).
If $X$ is smooth and projective, and if $n$ is invertible in $k$, one has an isomorphism $\mathrm{Br}\,X[n]\simeq H_{nr}^2(k(X)/k, \mu_n)$.

\paragraph{Acknowledgements.} This work was partially supported by NSF grant DMS-2201195. The author would like to thank J.-L. Colliot-Th\'el\`ene for several discussions and comments on the manuscript, and A. Auel for helpful discussions.

\section{The strategy for cubic surface bundles}\label{sstrategy}

\subsection{Cubic surface bundles}
Let $S$ be an integral scheme. Following \cite[Section 3]{ABP},
we say that $\pi:X\to S$ is a {\it cubic surface bundle} if $\pi$ is a flat projective morphism, the generic fibre $X_K$ of $\pi$ is a smooth cubic surface, and  $\pi_*\omega_{X/S}^{\vee}=E$ is a rank $4$ locally free $\mathcal O_S$-module such that $X\subset \mathbb P(E)$ is defined by the vanishing of a global section of $S^3(E^{\vee})\otimes L$ for some line bundle $L$ on $S$. 

We call {\it discriminant divisor} the locus $\Delta\subset S$  over which the fibres of $\pi$ are singular. Following \cite[Theorem 9]{ABP}, we will be interested in points $b\in S^{(1)}$ in $\Delta$ such that $X_b$ is irreducible over $\kappa(b)$, union of three planes over $\overline{\kappa(b)}$ permuted cyclically by the Galois group of a cyclic extension of $\kappa(b)$ of degree $3$.

\subsection{Strategy}\label{sstr}
In this paragraph we briefly review the general method to construct unramified classes for fibrations and we apply it to cubic surface bundles $\pi:X\to S$. This strategy was introduced by Colliot-Th\'el\`ene and Ojanguren in \cite{CTO} and it was used recently to obtain formulas for the unramified Brauer group for fibrations in conics and quadric surfaces over  rational surfaces (see \cite{ABBP, CT15,P}), and to construct more unramified classes for fibrations in higher-dimensional quadrics and in Fermat-Pfister forms (see  \cite{Sch2, Sch1, Sch}). See also \cite{KT1, KT2} for general formulas for fibrations in Severi-Brauer surfaces or involution surface bundles.

Let $k$ be an algebraically closed field such that $n$ is invertible in $k$.  We assume that $S$ is a smooth projective surface over $k$. We also assume that $S$ is  rational: this property is needed in step $3$ below.
Let $K$ be the field of functions of $S$ and let $X_K$ be the generic fibre of a cubic surface bundle $\pi:X\to S$.  Since $\mu_n\subset k$  we will identify $\mu_n$ and $\mathbb Z/n\mathbb Z$ as Galois modules.   The strategy is the following.

\begin{enumerate}
\item As in \cite{CTO}, the starting point is to look at classes in $\mathrm{Br}\,K$ which become unramified in the field of functions $k(X)=K(X_K)$: 
we  are interested in the intersection
$$\mathrm{Im}[\mathrm{Br}\,K\to \mathrm{Br}\,k(X)]\bigcap H^2_{nr}(k(X)/k, \mathbb Z/n\mathbb Z).$$
Following \cite{ABP}, for cubic surface bundles we are interested in the case $n=3$, and in the $3$-torsion in the Brauer group.

\item Recall that if $v$ is a valuation on $k(X)$ with the valuation ring $A\subset k(X)$, the {\it center} of $v$ on $S$ is defined as the image  $x_v\in S$ of the closed point of $Spec \,A$ under the composition of the morphism $Spec \,A\to X$ with $\pi:X\to S$, so that one has a homomorphism of local rings $\mathcal{O}_{S,x_v}\to A$.

Let $k(X)_v$ be the completion of the field $k(X)$ at $v$. Our crucial point is to make a stronger assumption for valuations $v$ such that the center $x_v$ of $v$ is not the generic point of $S$. Denote by $K_{x_v}$ the field of fractions of the completed local ring $\widehat{O}_{S,x_v}$. We will look at classes $\alpha$ in $H^2(K,  \mathbb Z/3\mathbb Z)$ such that the image of $\alpha$ in  $K_{x_v}$  lies in the kernel of the map
$$H^2(K_{x_v},  \mathbb Z/3\mathbb Z) \to H^2(K_{x_v}(X_K), \mathbb Z/3\mathbb Z).$$
This assumption is stronger than the assumption $\partial_v(\alpha)=0$. Indeed, 
one has the following commutative diagram of field extensions:

$$
\begin{tikzcd}
K(X_K) \arrow[hook]{r}&  K_{x_v}(X_K)\arrow[hook]{r} & K(X_K)_v=k(X)_v\\
K\arrow[hook]{r}\arrow[hook]{u}& K_{x_v}.\arrow[hook]{u} &
\end{tikzcd}
$$

Since the residue map (\ref{defres}) factorizes through $k(X)_v$,  one deduces  $\partial_v(\alpha)=0$.

This additional assumption is crucial in order to use the refinement introduced by Schreieder (see e.g. \cite{Sch, Sch1}) which allows to avoid any analysis of singularities in the specialization arguments.

\item  Let $\alpha\in H^2(K,  \mathbb Z/3\mathbb Z)$. Since $S$ is a rational surface, the class $\alpha$ is uniquely determined by its residues $\{\partial_x(\alpha)\}_{x\in S^{(1)}}$ at codimension $1$ points of $S$  (see for example \cite[Thm.1]{AM}). In order to obtain the general formula, we need to determine for which points $x=x_v$ the residue of $\alpha$ at $v$ is allowed to be nonzero.
As described above, the key ingredient is to understand the kernel 
$$\mathrm{Ker} [H^2(F,   \mathbb Z/3\mathbb Z) \to H^2(F(X_K), \mathbb Z/3\mathbb Z)]$$ for the following fields: \begin{itemize}
\item  for $F=K$ the function field of $S$, which will give us the image of $H^2(K,  \mathbb Z/3\mathbb Z)$ in $H^2(k(X), \mathbb Z/3\mathbb Z)$ in the first step above;
\item  for $F=K_x$ the field of fractions of the completed local ring $\widehat{O}_{S,x}$  at points $x=x_v$ of $S$ of codimension $1$ or $2$, which will allow us to achieve the second step. 
\end{itemize}
For fibrations in quadrics or Fermat-Pfister forms this analysis is done using the Milnor $K$-theory or Arason-Pfister theory and properties of quadratic forms. For cubic surfaces a more geometric approach apear to be working better. Following \cite{CTKM}, the properties of interest are: if the cubic surface is minimal or if it is $F$-birational to a Severi-Brauer surface (see Theorem \ref{formula} below).
\end{enumerate}

\subsection{Relative unramified cohomology}

Given the above strategy, we introduce the following notation:

\begin{defi}\label{hnrpi}
Let $k$ be a field and let $n>0$ be an integer invertible on $k$.  Let $S$ be a smooth projective integral variety over $k$ and let $K$ be the field of functions of $S$.  Let $X$ be an integral projective variety and let $\pi:X\to S$ be a dominant morphism, let $X_K$ be the generic fibre of $\pi$.
We define the group 
$$H^i_{nr, \pi}(k(X)/k, \mu_n^{\otimes j})\subset H^i(k(X), \mu_n^{\otimes j})$$
as the following intersection 
\begin{multline*}
H^i_{nr, \pi}(k(X)/k, \mu_n^{\otimes j})= \\
=\mathrm{Im} [H^i(K, \mu_n^{\otimes j})\to H^i(k(X), \mu_n^{\otimes j})]\bigcap \cap_x \mathrm{Ker}[H^i(k(X), \mu_n^{\otimes j})\to H^i(K_x(X_K), \mu_n^{\otimes j})],
\end{multline*}
 where $x$ runs over all scheme points of $S$ of positive codimension: $x\in S^{(i)}\mbox{ for }i>0$,  and where $K_x$ is the field of fractions of the completed local ring $\widehat{O}_{S,x}$  at $x$.

\end{defi}

\begin{rema}
From the definition, one has:
\begin{equation}\label{incl}
H^i_{nr, \pi}(k(X)/k, \mu_n^{\otimes j})\subset H^i_{nr}(k(X)/k, \mu_n^{\otimes j}).
\end{equation}
Indeed, let $\alpha\in H^i_{nr, \pi}(k(X)/k, \mu_n^{\otimes j})$ and let $v$ be a discrete valuation on $k(X)$. If $v$ is trivial on $K$, then $\partial_v(\alpha)=0$ since $\alpha$ is in the image of the group $H^i(K, \mu_n^{\otimes j})$. Otherwise, let $x_v\in S$ be the center of $v$. As in section \ref{sstr} above,  since $K_{x_v}(X_K)$ is a subfield of the completion $k(X)_v$, and since the residue map $\partial_v$ factorizes through the completion, one deduces $\partial_v(\alpha)=0$ again.
\end{rema}

The definition of the group $H^i_{nr, \pi}$ is inspired by the refined specialization method \cite{Sch1}: in particular, we show below that the classes in $H^i_{nr, \pi}$  satisfy the assumption of the following result of Schreieder.

\begin{prop}\label{Schref} (\cite[Proposition 3.1]{Sch1})
Let $X$ be a proper geometrically integral variety over a field $L$ which degenerates to a proper variety $Y$ over an algebraically closed field $k$. Let $\ell$ be a prime different from $char(k)$ and let $\tau:Y'\to Y$ be an alteration whose degree is prime to $\ell$. Suppose that for some $i\geq 1$ there is a nontrivial class $\alpha\in H^i_{nr}(k(Y/k, \mathbb Z/\ell\mathbb Z)$ such that
$$(\tau^*\alpha)|_E = 0\in H^i(k(E), \mathbb Z/\ell\mathbb Z)\mbox{ for any subvariety }E\subset \tau^{-1}(Y^{sing}),$$ where $Y^{sing}$ denotes the singular locus of $Y$.
Then $X$ is not stably rational over $F$.
\end{prop}

\begin{rema}
As in \cite[section 2.6]{Sch1} we say that $X$ degenerates to $X_0$ if there is a flat proper scheme $\mathcal X$ over a discrete valuation ring with the field of fractions $L$ and the residue field $k$, and an inclusion $L\subset F$  such that one has isomorphisms $X\simeq \mathcal X_{L}\times F$ and $X_0\simeq  \mathcal X_k$. 
\end{rema}

We deduce that a variety over an algebraically closed field with nonzero group $H^i_{nr, \pi}$ is a reference variety:

\begin{prop}\label{lrefv}
Let $k$ be an algebraically closed field and let $S$ be a smooth projective  integral variety over $k$. Let $X_0$ be an integral projective variety and let $\pi:X_0\to S$ be a dominant map with  smooth generic fibre.
Assume that there is an integer $i>0,$ and a prime number $\ell\neq char(k)$, such that
$$H^i_{nr, \pi}(k(X_0)/k, \mathbb Z/\ell\mathbb Z)\neq 0.$$
Let $X$ be a proper geometrically integral variety over a field $F$ which degenerates to $X_0$. Then $X$ is not stably rational over $F$.
\end{prop}

\begin{proof} Let $\alpha$ be a nonzero element in $H^i_{nr, \pi}(k(X_0)/k, \mathbb Z/\ell\mathbb Z)$. Let $K$ be the function field of $S$. Let $X_0^{sing}$ be the singular locus of $X_0$. We apply Proposition \ref{Schref}:  we show that for any alteration $\tau:X'\to X_0$ of degree prime to $\ell$, and for any
subvariety $E\subset \tau^{-1}(X_0^{sing})$ one has that the restriction of $\tau^*\alpha$ to $E$ is zero: $$\tau^*\alpha|_E=0\mbox{ in }H^i(k(E),\mathbb Z/\ell\mathbb Z).$$  

Note that $E$ does not dominate $S$ since the generic fibre $X_{0,K}$ of $\pi$ is smooth. In addition, up to blowing up $E$, one can assume that $E$ is a divisor (see e.g. the proof of  \cite[Proposition 5.1]{Sch1}). Let $v$ be the induced valuation on $X'$ and let $x_v$ be its center on $S$.  Let $A_v$ be the completion of the valuation ring of $v$. 

Since $\alpha$ is unramified, its image $\alpha_v$ in $H^i(k(X')_v, \mathbb Z/\ell\mathbb Z)$ comes from $$H^i_{\acute{e}t}(A_v, \mathbb Z/\ell\mathbb Z)\subset  H^i_{\acute{e}t}(k(X')_v, \mathbb Z/\ell\mathbb Z)$$ (see (\ref{gersten})). Recall that one then defines $\tau^*\alpha|_E$ as the image of $\alpha_v$ under the isomorphism $H^i_{\acute{e}t}(A_v, \mathbb Z/\ell\mathbb Z)\simeq H^i(k(E), \mathbb Z/\ell\mathbb Z)$.

By the definition of $H^i_{nr, \pi}$ we know that the image of $\alpha$ 
in $H^i(K_{x_v}(X_{0,K}),\mathbb Z/\ell\mathbb Z)$ is zero. 
Since $K_{x_v}(X_{0,K})\subset k(X')_v$, one deduces that $\alpha_v$ is the image of $\alpha$ by the composition  
$$H^i(k(X), \mathbb Z/\ell\mathbb Z)\to  H^i(K_{x_v}(X_K), \mathbb Z/\ell\mathbb Z)\to H^i(k(X')_v, \mathbb Z/\ell\mathbb Z),
$$
so that $\alpha_v=0$, hence $\tau^*\alpha|_E=0$ as well.
\end{proof}

\section{Function fields of cubic surfaces and example}\label{sff}
In this section we implement the first step of our strategy for cubic surface bundles (see section \ref{sstr}). 
Let $K$ be a field of characteristic $char(K)\neq 3$. 
Let $Y$ be a smooth projective cubic surface over $K$.
 We are interested in $3$-torsion elements in the group $\mathrm{Br}\,K[3]=H^2(K, \mu_3)$ which vanish  over $K(Y)$.

\begin{prop}\label{relbrc}
Let $K$ be a perfect field such that  $char(K)\neq 3$. Let $Y/K$ be a smooth projective cubic surface. Then the following are equivalent:
\begin{itemize}
\item [(i)]
$\mathrm{Ker} [H^2(K, \mu_3)\to H^2(K(Y), \mu_3)]\neq 0$;
\item [(ii)] there is a birational morphism $Y\to Y'$ such that $Y'$ is a non-split Severi-Brauer surface.
\end{itemize}
If these equivalent conditions hold, then
$$\mathrm{Ker} [H^2(K, \mu_3)\to H^2(K(Y), \mu_3)]\simeq \mathbb Z/3\mathbb Z,$$
generated by the class of $Y'$.
\end{prop}
\begin{proof}
Note that we are interested in the kernel  of the map $\mathrm{Br}\,K \to \mathrm{Br}\,K(Y)$, restricted to the $3$-torsion. 
This kernel was fully determined in \cite[Proposition 5.3]{CTKM} for all geometrically rational surfaces over a perfect field. In particular, it is a trivial, a $2$-torsion or a $3$-torsion group. The last case occurs only if $Y$ is not minimal, and there is a $K$-morphism $Y\to Y'$ to a non-split Severi-Brauer surface $Y'$, in which case the kernel is generated by the class of the surface.
\end{proof}

\

\begin{ex} \label{0symbol} 
 We give  two concrete examples:
 \begin{enumerate}
 \item 
 Assume that $\zeta_3\in K$. Let $a,b\in K^*$ and let $Y$ be the  following cubic surface  over $K$:
 \begin{equation}\label{birSB}
au^3+bv^3+abw^3+t^3=0. 
\end{equation}
Then the class  $(a,b)\in H^2(K, \mathbb Z/3\mathbb Z)$ vanishes over $K(Y)$ 
 (see \cite[Lemma 1]{CTKS}).
Indeed, if $a$ is a cube in $K(Y)$, then $(a,b)=0$. Otherwise, let $L=K(Y)(\sqrt[3]{a})$. Since in $K(Y)$ we have a relation
$$b=-\frac{t^3+au^3}{v^3+aw^3},$$ we have that $b=N_{L/K(Y)}(\beta)$ where $\beta=-\frac{t+\sqrt[3]{a}u}{v+\sqrt[3]{a}w}$.
Since  $a$ is a cube in $L$, we have the equality of symbols in $H^2(K(Y), \mathbb Z/3\mathbb Z)$:
$$(a,b)=(a, N_{L/K(Y)}(\beta))=N_{L/K(Y)}(a,\beta)=0,$$
where the second equality is the projection formula.
One could also write equations of $6$ disjoint lines on $Y$ which can be contracted over $K$. In particular, equation (\ref{birSB}) gives a birational model of the Severi-Brauer surface with the class $(a,b)$. 

\item Consider the following cubic surface $Y$: 
 \begin{equation}\label{cubic}
au^3+bv^3+abw^3+ft^3=0,
\end{equation}
where  $a,b,f\in K^*$ are  such that none of the elements $a,b,ab,f, af, bf$ is a cube in $K$. In this case, by Segre's theorem, the surface is minimal (see e.g. \cite[p.110]{M}), hence the map  $H^2(K, \mathbb Z/3\mathbb Z)\to H^2(K(Y), \mathbb Z/3\mathbb Z)$ is injective.
\end{enumerate}
\end{ex}

\

The above example allows to construct a cubic surface bundle which is a reference variety.

\begin{prop}\label{refv}
Let $k$ be an algebraically closed field of $char\, (k)\neq 3$ and let 
 $X\subset \mathbb P^2_{[x:y:z]}\times \mathbb P^3_{[u:v:w:t]}$ be a cubic surface bundle over $k$ given by the following equation
\begin{equation}\label{refcubic}
xz^2u^3+y^2zv^3+xy^2w^3+ft^3=0,
\end{equation}
where $$f=x^3+y^3+z^3+3x^2y+3xy^2+3y^2z+3yz^2+3xz^2+3x^2z,$$ and where $\pi:X\to \mathbb P^2_k$ is the projection on the first factor. Let $K=k(\mathbb P^2)=k(x/z,y/z)$, let $\alpha=(x/z, y/z)\in H^2(K, \mathbb Z/3\mathbb Z)$, and let $\alpha'$ be the image of $\alpha$ in $H^2(k(X), \mathbb Z/3\mathbb Z)$.
Then $\alpha'$ is not zero and  $$\alpha'\in H^2_{nr,\pi}(k(X)/k, \mathbb Z/3\mathbb Z).$$

\end{prop}

\begin{proof}
For convenience, we give a direct proof, which does not use the formula in Theorem \ref{formula}. First we observe that $\alpha'$ is nonzero. Indeed, we use Proposition \ref{relbrc} and the second example in \ref{0symbol}: the generic fibre $Y/K$ of $\pi$ is a minimal cubic surface.

Let $x_v\in \mathbb P^2_k$ be a point of positive codimension.  We have three cases:
\begin{enumerate}
\item Assume $x_v$ is the generic point of one of three lines $x=0$, $y=0$, or $z=0$, or an intersection point of two of these lines. Then, by the definition of $f$, we see that $f$ is a nonzero cube in the residue field $\kappa(x_v)$, so that $f$ is a cube in $K_{x_v}$ by the Hensel lemma. We then get that over $K_{x_v}$ the cubic surface $Y_{K_{x_v}}$ is of the first type considered in example \ref{0symbol}, so that the element $(x/z, y^2/z^2)=2\alpha$ is in the kernel of the map 
\begin{equation}
 H^2(K,\mathbb Z/3\mathbb Z) \to H^2(K_{x_v},\mathbb Z/3\mathbb Z)\to H^2(K_{x_v}(Y), \mathbb Z/3\mathbb Z).\label{sfactor}
 \end{equation}
Hence $\alpha'\in \mathrm{Ker}[H^2(k(X), \mathbb Z/3\mathbb Z)\to H^2(K_{x_v}(Y), \mathbb Z/3\mathbb Z)]$.

\

For the remaining cases, by the sequence (\ref{sfactor})  above, it is enough to show that the image of $\alpha$ in $K_{x_v}$ is zero:

\item  Assume $x_v$ is a closed point lying on only one of the lines $x=0$, $y=0$, or $z=0$. By symmetry we may assume that $x_v$ is on the line $x=0$. Then $y/z$ is a nonzero element in the residue field $\kappa(x_v)=k$, hence a cube since $k$ is algebraically closed. Hence $y/z$ is a cube in $K_{x_v}$ and the image of $\alpha$ in $K_{x_v}$ is zero.  
\item In all other cases  the elements $x/z$ and $y/z$ are units in the local ring of $x_v$, so that the image of $\alpha$ in $K_{x_v}$ comes from the cohomology  group 
$H^2_{\acute{e}t}( \widehat{\mathcal O}_{\mathbb P^2, x_v},\mathbb Z/3\mathbb Z)$ of the completed local ring, but this group is zero by cohomological dimension since $H^2_{\acute{e}t}( \widehat{\mathcal O}_{\mathbb P^2, x_v},\mathbb Z/3\mathbb Z)=H^2( \kappa(x_v),\mathbb Z/3\mathbb Z)=0$. 
\end{enumerate}

\end{proof}

\begin{rema}
For the example above one could also use the techniques of the universal relations in Milnor $K$-theory in \cite{Sch}. In particular, \cite[Theorem 5.3]{Sch} applies with $n=2$,  $x_1=x$, $ x_2=y^2$, $a_1=xu^3$, $a_2=y^2(xw^3+v^3)$, $b=-f(x,y,1)$.
\end{rema}

As an application, we recover a result of Krylov-Okada \cite[Theorem 1.2(3)]{KO}
 (see also \cite[Proposition 6.1]{NO} in zero characteristic, using techniques from tropical geometry):

\begin{coro}(\cite{KO, NO})\label{cNO}
Let $k$ be an algebraically closed field of $char\, (k)\neq 3$. A very general hypersurface of bidegree $(3,3)$ in $\mathbb P^2_k\times \mathbb P^3_k$ is not stably rational.
\end{coro}
\begin{proof}
This follows from Proposition \ref{lrefv} and Proposition \ref{refv}.
\end{proof}

\section{Local computations and the general formula}\label{sf}

In this section we finish implementing our strategy for a large class of cubic surface bundles $\pi:X\to S$: we analyze the kernels
$\mathrm{Ker}[H^2(K_x,   \mathbb Z/3\mathbb Z) \to H^2(K_x(X_K), \mathbb Z/3\mathbb Z)]$ for $x$ a point of codimension $1$ or $2$ of $S$, and we derive a formula for the group $H^2_{nr, \pi}$.

\subsection{The case of dimension $1$}

In this paragraph we assume that $K_v$ is a complete discretely valued field of equicharacteristic zero and that $\zeta_3\in K$. Let $A_v$ be the valuation ring, and let $\kappa(v)$ be the residue field.

\begin{prop}\label{cdcond1}
  Let $\mathcal Y/A_v$ be a flat cubic surface bundle, such that the generic fibre $Y=\mathcal Y_{K_v}$ is smooth, and the special fibre $Y_0=\mathcal Y_{\kappa(v)}$ is reduced. Assume there is a birational morphism $Y\to Y'$ such that $Y'$ is a non-split Severi-Brauer surface. Let $\alpha\in H^2(K_v, \mathbb Z/3\mathbb Z)$ be its class.  Assume that $cd(\kappa(v))\leq 1$.
Then
\begin{itemize}
\item [(i)] $Y_0$ is geometrically a union of three planes permuted cyclically by Galois, and $\partial_v(\alpha)= \xi$ or $\xi^{-1}$, where $\xi$ is the class of this extension;
\item[(ii)] if $\beta\in H^2(K_v, \mathbb Z/3\mathbb Z)$ is such that $\partial_v(\alpha)= \xi$ or $\xi^{-1}$, then $\beta=\pm \alpha$. 
\end{itemize}
\end{prop}
\begin{proof}
\begin{itemize}
\item [(i)]
Since the cohomological dimension $cd\,(\kappa(v))\leq 1$, the class $\alpha$ is uniquely determined by its residue $\partial_v(\alpha)$. Indeed, from sequence (\ref{gersten}), if  $\partial_v(\alpha_1)=\partial_v(\alpha_2)$, then $\alpha_1-\alpha_2$ comes from $H^2_{\acute{e}t}(A_v, \mathbb Z/3)=H^2(\kappa(v), \mathbb Z/3)=0$. In particular,  since $Y'$ is not split, one has that $\partial_v(\alpha)\neq 0$.

By Proposition \ref{relbrc}, the image $\alpha'$  of $\alpha$ in $K_v(Y)$ is zero; in particular it is unramified with respect to any discrete valuation $w$ on $K_v(Y)$ extending $v$. The conclusion (i) then follows from  \cite[Lemma 8, Theorem 9]{ABP}.
\item [(ii)]
Similarly as above, since $\partial_v(\beta)=\pm \partial_v(\alpha)$, one has $\beta=\pm \alpha$.
\end{itemize}

\end{proof}

\begin{rema}\label{resc}
By Proposition \ref{relbrc}, the assumption that there is a birational morphism $Y\to Y'$ such that $Y'$ is a non-split Severi-Brauer surface is equivalent to the assumption $\mathrm{Ker} [H^2(K_v, \mathbb Z/3\mathbb Z)\to H^2(K_v(Y), \mathbb Z/3\mathbb Z)]\neq 0$.
\end{rema}

\subsection{The case of dimension $2$}

We first recall the local description of Brauer classes on a surface from \cite{S}.
Let $S$ be a smooth projective surface over an algebraically closed field $k$ of characteristic zero. Let $x\in S$ be a closed point, let $A_x=\widehat{O}_{S,x}$ be the completed local ring at $x$, and let $K_x$ be the field of fractions of $A_x$. 

\begin{lemma}\label{hot}
Let $\alpha\in H^2(K_x, \mathbb Z/3\mathbb Z)$. Assume that  $$ram(\alpha)=\{\mathfrak{p}\in \mathrm{Spec}\,A_x \text{ of height }1, \partial_{\mathfrak{p}}(\alpha)\neq 0\}$$ is a simple normal crossings divisor. Then either $\alpha=0$ or 
$\alpha=\pm (\pi_1, \pi_2)$, where $\pi_1, \pi_2$ generate the maximal ideal of $A$.
\end{lemma}
\begin{proof}
The statement follows from \cite[Theorem 2.1]{S} and the observation that $H^i_{\acute{e}t}(A_x, \mathbb Z/3\mathbb Z)=H^i(\kappa(x), \mathbb Z/3\mathbb Z)=0$ for $i>0$ since $\kappa(x)=k$ is algebraically closed.
\end{proof}

\begin{prop}\label{cdcond}
 Let $\mathcal Y/A_x$ be a flat cubic surface bundle, such that the generic fibre $Y=\mathcal Y_{K_x}$ is smooth. Assume that the fibres $\mathcal Y_{\mathfrak{p}}$ over height one prime ideals $\mathfrak{p}\in \mathrm{Spec}\,A_x$ are reduced.
   Assume there is a birational morphism $Y\to Y'$ such that $Y'$ is a non-split Severi-Brauer surface. Let $\alpha\in H^2(K_x, \mathbb Z/3\mathbb Z)$ be its class. Let
   $$ram(\alpha)=\{\mathfrak{p}, \partial_{\mathfrak{p}}(\alpha)\neq 0\}$$
be the ramification divisor of $\alpha$. Then:
\begin{itemize}
\item[(i)]   for any $\mathfrak{p}\in ram(\alpha)$  the fibre $\mathcal Y_{\mathfrak{p}}$ is geometrically a union of three planes permuted cyclically by Galois, and  $\partial_v(\alpha)= \xi_{\mathfrak{p}}$ or $\xi_{\mathfrak{p}}^{-1}$, where $\xi_{\mathfrak{p}}$ is the class of this extension;
\item [(ii)] if $ram(\alpha)$ is a simple normal crossings divisor, then $\alpha$ is ramified at exactly two primes $\mathfrak{p}_1, \mathfrak{p}_2$.
\end{itemize}
\end{prop}
\begin{proof}
For $(i)$, we apply Proposition \ref{cdcond1} to the (completion of) localisation of $A_x$ at $\mathfrak{p}\in ram(\alpha)$ (note that since $\alpha$ is ramified at $\mathfrak p$, it remains nonzero). 
For $(ii)$ we apply Lemma \ref{hot}.
\end{proof}

\begin{rema}\label{unicity}
From codimension one purity for $A_x$ (see \cite[Theorem 3.8.3]{CT}) and the fact that $H^2_{\acute{e}t}(A_x, \mathbb Z/3\mathbb Z)=0$, one has  that $\alpha$ is uniquely determined by the set of residues $\{\partial_{\mathfrak{p}}(\alpha)\}$ for $\mathfrak{p}\in ram(\alpha)$.
\end{rema}

\subsection{A general formula for fibrations in cubic surfaces}

 For simplicity in this section we assume that $k$ is of characteristic zero.

\begin{thm}\label{formula}
Let $k$ be an algebraically closed field of characteristic zero, let $S$ be a smooth projective rational surface over $k$, and let $\pi:X\to S$ be a cubic surface bundle with discriminant divisor $\Delta\subset S$. Let $K$ be the function field of $S$. Assume that the generic fibre $X_K$ is a smooth minimal cubic surface and that the fibres of $\pi$ over codimension $1$ points of $S$ are reduced.  Let $C=\cup_{i=1}^n C_i\subset \Delta$ be a divisor corresponding to the set of codimension $1$ points of $S$ over which the fibre of $\pi$ is geometrically a union of three planes permuted cyclically by Galois, with $\gamma_i\in \kappa(C_i)^*/(\kappa(C_i)^*)^3$ the class corresponding to the cyclic extension.
 Assume that $C$ is a simple normal crossings divisor on $S$. 
Then 
 $$H^2_{nr, \pi}(k(X)/k, \mathbb Z/3\mathbb Z)\subset (\mathbb Z/3\mathbb Z)^n$$
 is the subgroup of elements $\underline a=\{a_i\}_{i=1}^n$, $a_i\in \{-1,0,1\}$ satisfying the following properties:
 \begin{itemize}
\item[(i)] if $a_i\neq 0$, the base change $X_{K_{x_i}}$ of the surface $X_K$ to the completion $K_{x_i}$ of $K$ at the generic point $x_i$ of $C_i$ is not minimal: there is a birational morphism $X_{K_{x_i}}\to X'_i$ such that $X'_i$ is a non-split Severi-Brauer surface over $K_{x_i}$;
\item[(ii)] the sum $$\sum_{i=1}^n\sum_{P\in S^{(2)}} \partial_P(\gamma_i^{a_i})=0$$ is zero, and for every point  $P\in C_i\cap C_j$  such that $\partial_P(\gamma_i^{a_i})=-\partial_P(\gamma_j^{a_j})\neq 0$, one has that  the base change $X_{K_{P}}$ of the surface $X_K$ to the field of fractions $K_P$ of the completed local ring $\widehat{\mathcal{O}}_{S,P}$  is not minimal: there is a birational morphism $X_{K_P}\to X'_P$ such that $X'_P$ is a non-split Severi-Brauer surface over $K_P$. 
\end{itemize}

\end{thm}

\begin{proof}
We first check that any vector $\underline a=\{a_i\}_{i=1}^n$ satisfying properties $(i)$ and $(ii)$ gives a class $\alpha\in H^2_{nr, \pi}(k(X)/k, \mathbb Z/3\mathbb Z)$.

Since $S$ is a smooth projective rational surface, the Bloch-Ogus complex
$$0\to H^2(K,\mathbb Z/3\mathbb Z) \stackrel{\oplus \partial^2}{\to}  \oplus_{x\in S^{(1)}} H^1(\kappa(x),\mathbb Z/3\mathbb Z) \stackrel{\oplus \partial^1}{\to} \oplus_{P\in S^{(2)}} H^0(\kappa(P),\mathbb Z/3\mathbb Z)$$
is exact (see for example \cite[Thm.1]{AM} ). 
Hence there is a unique class $\alpha\in H^2(K,\mathbb Z/3\mathbb Z)$ such that the ramification divisor of $\alpha$ is $\sum_{i} |a_i|C_i$, and $\partial_{C_i}(\alpha)=\gamma_i^{a_i}$, $1\leq i\leq n$.

Let $\alpha'$ be the image of $\alpha$ in $H^2(k(X), \mathbb Z/3\mathbb Z)$. We show that $\alpha'$ is a nonzero element in $H^2_{nr,\pi}$.  First, since $X_K$ is minimal, the map $$H^2(K,\mathbb Z/3\mathbb Z)\to H^2(K(X_K),\mathbb Z/3\mathbb Z)= H^2(k(X),\mathbb Z/3\mathbb Z)$$ is injective by Proposition \ref{relbrc}, hence $\alpha'$ is nonzero. Let $x\in S$ be a point of codimension $1$ or $2$. Let $\alpha_x$ be the image of $\alpha$ in $H^2(K_{x}, \mathbb Z/3\mathbb Z)$. We have two cases:

\begin{enumerate}
\item Assume that $x$ is the generic point of a curve $D\subset S$:
\begin{enumerate}
\item  If $D$ is not one of the ramification curves $C_i$ of $\alpha$, then $\partial_{D}(\alpha)=0$ and the image $\alpha_x$ of $\alpha$ in  $H^2(K_{x}, \mathbb Z/3\mathbb Z)$ comes from an element in the cohomology of the completed local ring $H^2_{\acute{e}t}(\hat{\mathcal O}_{S, x}, \mathbb Z/3\mathbb Z)$ (see (\ref{gersten})), but $H^2_{\acute{e}t}(\hat{\mathcal O}_{S, x}, \mathbb Z/3\mathbb Z)=H^2(\kappa(D), \mathbb Z/3\mathbb Z)=0$  by  cohomological dimension.  
\item Assume now $D$ is one of the ramification curves $C_i$ of $\alpha$. By the definition of $\alpha$ and Proposition  \ref{cdcond1} one has that $\alpha_x$ or $-\alpha_x$  equals to the class of the Severi-Brauer surface $X_i'$, hence it is zero in $K_{x}(X_K)$ by Proposition \ref{relbrc}.
\end{enumerate}

\item Assume that $x=P$ is a closed point of $S$. Since $ram(\alpha_x)$ is a simple normal crossings divisor by construction we can apply Lemma \ref{hot}: either $\alpha_x$ is zero, or $\alpha_x=\pm (\pi_i, \pi_j)$ where $\pi_i$ and $\pi_j$ are local parameters of the ramification curves $C_i$ and $C_j$.  In the latter case one has  $\partial_{P}(\partial_{C_i}(\alpha_x))=\pm 1\neq 0$. Hence by  assumption  there is a birational morphism $X_{K_P}\to X'_P$ such that $X'_P$ is a non-split Severi-Brauer surface over $K_P$. Let $\beta$ be the class of this surface. Since $\beta$ is nonzero and the fibres of $X\to S$ are geometrically unions of three planes permuted by Galois only at $C_i$ and $C_j$, but not at other curves passing by $P$, by Proposition \ref{cdcond} one has that $\beta$ ramifies exactly at $C_i$ and $C_j$ and the residues of $\beta$ are $\pm \gamma_i, \pm\gamma_j$. Since in addition $\partial_{P}(\partial_{C_i}(\beta))+ \partial_{P}(\partial_{C_j}(\beta))=0$, one deduces that $\beta$ and $\alpha_x$ have the same or opposite residues, so that, by remark \ref{unicity}, one has that $\beta=\pm \alpha_x$. Hence $\alpha_x$ is zero in $K_P(X_K)$ by Proposition  $\ref{relbrc}$ again.

\end{enumerate}

Hence a vector $\underline{a}$ satisfying the assumptions gives a nonzero class in the group $H^2_{nr,\pi}(k(X)/k, \mathbb Z/3\mathbb Z)$.

For the converse, let $\alpha'\in H^2_{nr,\pi}(k(X)/k, \mathbb Z/3\mathbb Z)$. Then $\alpha'$ is the image of an element $\alpha\in H^2(K, \mathbb Z/3\mathbb Z)$. Let  $D\in ram(\alpha)$ and let $x$ be the generic point of $D$. By assumption, the image of $\alpha$ in $H^2(K_{x}(X_K),\mathbb Z/3\mathbb Z)$ is zero. Hence, by Proposition \ref{relbrc}, there is a birational map $X_{K_{x}}\to X'$ such that $X'$ is a non-split Severi-Brauer surface. By Proposition \ref{cdcond1}, one has that $D$ is one of the curves $C_i$, and $\partial_{C_i}(\alpha)=\gamma_i$ or $\gamma_i^{-1}$. We then define $a_i=1$ or $-1$ if, respectively, $\partial_{C_i}(\alpha)=\gamma_i$ or $\gamma_i^{-1}$  and we obtain the property  $(i)$.  

From the Bloch-Ogus complex in the beginning of the proof we deduce that $$\sum_{i=1}^n\sum_{P\in S^{(2)}} \partial_P(\gamma_i^{a_i})=0.$$ Let  $P\in C_i\cap C_j$  be such that $\partial_P(\gamma_i^{a_i})=-\partial_P(\gamma_j^{a_j})\neq 0$. 
In particular, the image of $\alpha$ in $H^2(K_P, \mathbb Z/3\mathbb Z)$ is not zero, and it is in the kernel of the map $$H^2(K_P, \mathbb Z/3\mathbb Z)\to H^2(K_P(X_K), \mathbb Z/3\mathbb Z)$$ since $\alpha'$ is in $H^2_{nr,\pi}$.  Then $(ii)$ follows from Proposition \ref{relbrc} again.

\end{proof}

The above formula allows to construct unramified classes effectively, given a cubic surface bundle $X\to S$:
\begin{enumerate}
\item  we check that $X_K$ is a minimal cubic surface, we determine the discriminant locus $\Delta$ (see e.g. \cite{E} for an explicit formula), and we verify that the fibres of $\pi$ over generic points of $\Delta$ are reduced;
\item  we determine $C\subset \Delta$ corresponding to codimension one points of $S$ over which the fibre of $\pi$ is geometrically a union of three planes permuted by Galois, we check that $C$ is a simple normal crossings divisor, and we record the corresponding classes $\gamma_i$;
\item we determine for which irreducible curves $C_i$  in $C$ the surface $X_{K_i}$ is not minimal, and $K_i$-birational to a Severi-Brauer surface, we call $C'\subset C$ the corresponding divisor; 
 \item we keep linear combinations $\sum a_i C_i$ with $C_i$ in $C'$ and $a_i\in \{-1,0,1\}$  satisfying condition $(ii)$ of Theorem \ref{formula}.
\end{enumerate}

\

{\footnotesize
\noindent \textsc{Alena Pirutka, Courant Institute of Mathematical Sciences, \newline
New York University, 
New York, U.S.A.}
\newline
\textit{pirutka@cims.nyu.edu}
}

\end{document}